\documentclass[11pt]{amsart}
\usepackage{amssymb}
\usepackage{amsfonts}
\usepackage{amsmath}
\usepackage{graphicx}

\usepackage{xcolor}
\usepackage{amsmath}
\usepackage{mathrsfs}
\usepackage{stmaryrd}
\usepackage{epsfig,color}
\usepackage{blindtext}
\usepackage{enumerate}
\usepackage{hyperref}
\usepackage{url}
\usepackage{bbm}
\usepackage{filecontents}
\usepackage{nicefrac,mathtools}
\usepackage{bm}   
\DeclareGraphicsExtensions{.pdf,.jpeg,.png}
\usepackage{epstopdf}
\usepackage{cancel} 
\usepackage[normalem]{ulem} 
\usepackage{verbatim} 
\usepackage{enumitem} 
\usepackage{color}
\usepackage[msc-links, lite]{amsrefs}
\usepackage{geometry}
\newtheorem{theorem}{Theorem}[section]
\newtheorem{conjecture}[theorem]{Conjecture}
\newtheorem{proposition}[theorem]{Proposition}
\newtheorem{lemma}[theorem]{Lemma}

\newtheorem{example}[theorem]{Example}

\newtheorem{claim}[]{Claim}

\theoremstyle{definition}
\newtheorem{definition}[theorem]{Definition}

\numberwithin{equation}{section}

\newcommand{\mb}{\mathbb}
\newcommand{\mc}{\mathcal}

\newcommand{\mk}{\mathfrak}

\newcommand{\wti}{\widetilde}

\newcommand{\mr}{\mathrm}
\newcommand{\dist}{\operatorname{dist}}

\DeclareMathOperator{\Ric}{Ric}

\title{Uryson width of three dimensional mean convex domain with non-negative Ricci curvature}
\date{\today}
\author{Zhichao Wang}
\address{Department of Mathematics, The University of British Columbia, Vancouver, BC V6T 1Z2, Canada}
\email{zhichao@math.ubc.ca}

\author{Bo Zhu}
\address{School of Mathematics  University of Minnesota-Twin Cities, MN 55455}
\email{zhux0629@umn.edu}
\begin{document}
	
	\begin{abstract}
		We prove that for every three dimensional manifold with nonnegative Ricci curvature and strictly mean convex boundary, there exists a Morse function so that each connected component of its level sets has a uniform diameter bound, which depends only on the lower bound of mean curvature. This gives an upper bound of  Uryson 1-width for those three manifolds with boundary.
	\end{abstract}

	\maketitle
	
	\section{Introduction}
	In \cite{Gro19}, Gromov proposed the following conjecture:
	\begin{conjecture}[Gromov \cite{Gro19} ] \label{conjecture}
		Suppose that  $X \subset  \mathbb{R}^n$ is a smooth domain such that $H_{\partial X} \geq n-1$.
		Then there exists a continuous self-map $R: X \rightarrow  X$ such that
		\begin{itemize}
			\item  the image $R(X) \subset  X $ has topological dimension $n -2$;
			\item $\dist(x, R(x)) \leq  c_n $ for all $x \in  X$, with the best expected $c_n=1$.
		\end{itemize}
	\end{conjecture}
	
	Recall that the {\em Uryson $k$-width $\mathrm{width}_k(M)$} of a Riemannian manifold $M$ is the infimum of the real numbers $d\geq 0$, such that there exist a $k$-dimensional polyhedral space $P^k$ and a continuous map $f:M\to P^k$ with 
	\[ \mathrm{diam}_Mf^{-1}(p)\leq d, \ \ \ \text{ for all } \  p\in P^k,\]
	where $\mr{diam}(\cdot)$ denotes the diameter of the subset of $M$. Clearly, Conjecture \ref{conjecture} implies that the Uryson $(n-2)$-widths of mean convex domains in Euclidean spaces are bounded from above by a constant relying on their mean curvature lower bounds. In this paper, we we give a direct proof of such an upper bound. More generally, our result holds for all three-dimensional mean convex domains with non-negative Ricci curvature.
	\begin{theorem}\label{levelsetcontrol}
		Suppose that $(M,\partial M,g)$ is a complete (possibly non-compact) three dimensional Riemannian manifold with $\Ric(g)\geq 0$ and $H_{\partial M}\geq 1$. Then there exists a smooth Morse function $f: M\to \mb R$ such that for any $t$ and  $x,y$ in the same connected component of $f^{-1}(t)$,
		\[ \dist_M(x,y)< 117.\]
		In particular, if $M$ is a smooth domain in $\mb R^3$ with $H_{\partial M}\geq 1$, then the upper bound can be improved to be $49$.
	\end{theorem}

	We remark that the condition of non-negative Ricci curvature can not be relaxed to non-negative scalar curvature due to the following example.
	
	\begin{example}
		Let $B_R(x)$ be the Euclidean ball in $\mb R^3$ centered at $x$ with radius $R$. Then $(\mathbb{R}^3\setminus B_{1/3}(0),g)$ is a Riemannian manifold with positive scalar curvature, where \[g_{ij}=(4+\frac{4}{r})^4\delta_{ij}, \ r > 0.\]
		Here $r$ is the distance function to the origin with respect to the Euclidean metric $\delta$. Moreover, its boundary $\partial B_{1/3}(0)$ has mean curvature $H=3>1$ with respect to the outward normal vector field.  However, outside a sufficiently large ball, the manifold is close to the Euclidean spaces, which has infinite Uryson 1-width.
	\end{example}
	
	Constructing a singular foliation by surfaces of controlled size has been successfully used to understand the structure of three dimensional manifolds with positive scalar curvature. Gromov-Lawson \cite{GLPSC} obtained an upper bound of Uryson 1-width for simply connected Riemannian manifolds by considering the level sets of distance function to a fixed point. For closed manifolds with nonnegative Ricci and positive scalar curvature, Marques-Neves \cite{MN12} proved a sharp bound on the area of the maximal leaves. In a recent work \cite{LM20}, Liokumovich-Maximo proved that every closed three manifold with positive scalar curvature admits singular foliations so that each leave has controlled diameter, area and genus. In \cite{Gro19_2}*{\S 3.10, Property A}, Gromov proved the Uryson 1-width upper bound for three dimensional complete (possibly non-compact) Riemannian manifolds with positive scalar curvature. Our main result in this paper is an analog of Gromov's result for compact manifolds with non-negative Ricci and strictly mean convex boundary, which also implies a new proof for \cite{Gro19_2}*{\S 3.10, Property A}.

	\subsection*{Challenges and ideas}The main challenge in our paper is to decompose the manifold into {\em geometrically prime regions}, i.e. those regions where each closed curve bounds a surface relative to a connected boundary component. The idea is to cut the manifold along two-sided stable free boundary minimal surfaces. Unlike the argument in \cite{LM20}, we don't have a bumpy metric theorem for non-compact manifold with the non-negative Ricci curvature preserved. Nevertheless, we can find countably many stable free boundary minimal surfaces so that after cutting along them, each connected component contains only ``trivial'' two-sided stable ones which are isotopic to one of the boundary components. Then for those connected components that are not geometrically prime, we are going to apply min-max theory in the ``core region'' (see \cite{Song19} by A. Song) to find a index one free boundary minimal surfaces that subdivides the the components into two geometrically prime regions.
	
	However, this ``core region'' could be non-compact and there is no general min-max theory for such manifolds. In this paper, to deal with non-compact manifolds, we take a sequence of compact domains that exhaust the manifold; c.f. \cite{Song19}*{\S 3.2}. We perturb the metric in the neighborhood of the new boundary so that the new boundary component becomes a stable free boundary minimal surface. However, the perturbation will also produce more stable free boundary minimal surfaces. More importantly, the diameter bounds for stable/index one surfaces can not be preserved anymore since the Ricci curvature will not be non-negative with respect to the new metric. Fortunately, as these compact domains exhausting the non-compact manifold, the new stable free boundary minimal surfaces are far away from a fixed compact domain. Then by cutting along those surfaces with small area in a suitable order, one can obtain a sequence of compact ``core regions'' converging to the non-compact domain in the sense of Gromov-Hausdorff. Moreover, these ``core regions'' satisfy a weak Frankel property, which is directly from cutting process. By applying the min-max theory to these compact ``core regions'', one can construct a sequence of two-sided free boundary minimal surfaces with index one. By the weak Frankel property, these min-max surfaces should intersect a given domain, which implies the limit of this sequence of surfaces is non-empty. Hence such a sequence of surfaces are actually free boundary minimal surfaces with respect to the original metric. This gives the desired surfaces.
	
	\medskip
	The remaining issue is to adapt Gromov-Lawson' trick to each geometrically prime region. Then we require a uniform diameter bound for the two-sided stable/index one free boundary surfaces that we have cut. Recall that the length of boundaries of these surfaces have been bounded by Ambrozio-Buzano-Carlotto-Sharp \cite{ABCS}*{Lemma 48}, which is still far from the diameter bound. In this paper, we obtain a general {\em radius bound} (i.e. distance bound from interiors to boundaries) for any smooth surfaces with bounded mean curvature. Suppose on the contrary that there exists a surface-with-boundary that has large radius. Then regarding such a surface as a barrier, by a minimizing process in some relative homology class, there is a stable constant mean curvature surface with mean curvature 1, whose possible boundaries are far away from a interior point. By applying Schoen-Yau's trick \cite{SY83} here, such a cmc surface has a uniform radius bound, which implies that it is closed. Clearly, in Riemannian manifolds with non-negative Ricci curvature, there is no closed stable cmc surface. Such a contradiction gives the radius bound for any surface-with-boundary. Combining with the length bound of boundaries of free boundary minimal surfaces with index one, we then obtain the diameter upper bounds for these surfaces.

	\subsection*{Outline}This paper is organized as follows. In Section \ref{sec:pre}, we prove a ``radius'' bound for each embedded surface with bounded mean curvature. Then combining the length estimates in \citelist{\cite{ABCS}\cite{CF20}}, we obtain a diameter upper bound for two-sided free boundary minimal surface with index less than or equal to 1. In the second part of this section, we state the diameter estimates for the level sets of distance functions in geometrically prime regions. In Section \ref{sec:decomp}, we decompose three dimensional manifolds with nonnegative Ricci and strictly positive mean  curvature into countably many geometrically prime regions. The most technical part is Proposition \ref{prop:cut two side to get prime}, where we will use the min-max theory to produce free boundary minimal surfaces of index one. Finally, in Section \ref{sec:bound}, we construct the desired function in each geometrically prime region and then glue them together to get the desired function. For the sake of completeness, we adapt Gromov-Lawson's trick in Appendix \ref{sec:GL trick}, which is parallel to \cite{LM20}*{Lemma 4.1}.

	\subsection*{Acknowledgement} We would like to thank Professors Yevgeny Liokumovich and Xin Zhou for many helpful discussions. Z.W. would like to thank Professors Ailana Fraser and Jingyi Chen for their interests. B.Z. would like to express his deep gratitude to his advisor Jiaping Wang for his expert guidance and encouragement.

	\section{Preliminaries}\label{sec:pre}
	
	\subsection{Stable free boundary minimal surfaces}
	In this subsection, we will prove a diameter upper bound for stable minimal surfaces in a complete, three dimensional manifold $(M,\partial M,g)$ with nonnegative Ricci curvature and strictly mean convex boundary. Note that  {by} \citelist{\cite{CF20}*{Proposition 1.8} \cite{ABCS}*{Lemma 48}}, the length of boundary of a two-sided stable free boundary minimal surface is bounded even if $M$ is non-compact. Therefore, the diameter upper bound of these surfaces would follow from an upper bound on the distance to their boundaries for all interiors.

	Recall that Schoen-Yau \cite{SY83} proved a diameter upper bound for stable minimal surfaces in three dimensional manifolds with strictly positive scalar curvature. By adapting their arguments, one can obtain a diameter upper bound for stable constant mean curvature (cmc) surfaces in three dimensional manifolds with non-negative scalar curvature. We state the result here and refer to \cite{LZ18}*{Proposition 2.2} for more details.
	
	\begin{proposition}\label{lem:bound radius of cmc}
		Let $(M^3,\partial M,g)$ be a three dimensional Riemannian manifold with nonempty boundary and scalar curvature $R_M \geq 0$ and $\Sigma$ be a connected, embedded, compact stable cmc surface with mean curvature $H>0$. Then for any $x \in \Sigma$,
		\begin{equation}\label{diameter}
		\dist_{\Sigma}(x,\partial \Sigma)\leq \frac{4\pi}{3 H}.
		\end{equation}
	\end{proposition}
	
	Observe that each surface with mean curvature bounded by 1 can be a one-sided barrier for constant mean curvature surface with $H=1$. Thus we can construct a stable cmc surface by a minimizing process if there is a ``large'' surface-with-boundary having bounded mean curvature. Then Proposition \ref{lem:bound radius of cmc} implies the closeness of such a minimizer which contradicts the non-negative Ricci curvature. 
	
	\begin{theorem}\label{thm:compact boundary}
		Let $(M^3,\partial M,g)$ be a compact Riemannian manifold with nonempty boundary. Suppose that $\Ric\geq 0$ and $H_{\partial M}\geq 1$. Let $\Sigma$ be an embedded surface with $|H_{\Sigma}|\leq 1$. Then 
		\[  \sup_{x\in\Sigma} \dist_M(x,\partial \Sigma)\leq \frac{4\pi}{3}+2.\]
		In particular, $\Sigma$ is compact if and only if its boundary is compact.
	\end{theorem}
	\begin{proof}
		Suppose not, then there exist $p\in\Sigma$ and $\epsilon>0$ such that
		\[ \dist_{M}(p,\partial \Sigma)> \frac{4\pi}{3}+2+2\epsilon.\]
		Let $M_1=M\cap B(p,\frac{4\pi}{3}+2+2\epsilon)$ and $T$ denotes the closure of $\partial M_1\cap \mathrm{Int}M$. Here we may assume that $T$ is transverse to $\partial M$. 
		
		Now we cut $M_1$ along $\Sigma$ and denote by $M_2$ the metric completion of $M_1\setminus \Sigma$. So long as $\Sigma$ separates $M_1$, we choose one of the connected components of $M_1\setminus \Sigma$, still denoted by $M_2$. Then we set $\Sigma'=\partial M_2\setminus \partial M_1$, which belongs to one of the following:
		\begin{itemize}
			\item $\Sigma'$ is a double cover of $\Sigma$;
			\item $\Sigma'$ is diffeomorphic to $\Sigma$;
			\item $\Sigma'$ has two connected components and each component is diffeomorphic to $\Sigma$. 
		\end{itemize}
		In each case, there exists a point (still denoted by $p$) so that 
		\begin{equation}\label{eq:p to T}
		\dist_{M_2}(p,T\cap M_2)>\frac{4\pi}{3}+2+2\epsilon.
		\end{equation} 
		Note that by Ricci comparison theorem \cite{Pet}*{Chapter 9, Proposition 39}, $\dist_M(p,\partial M)\leq 2$ since $M$ has nonnegative Ricci curvature. Then there exists a smooth curve $\gamma:[0,1]\to M_2$ with
		\begin{equation}\label{eq:gamma to p}
		\gamma(0)=p, \ \ \gamma(1)\in M_2\cap\partial M \ \ \text{ and } \ \ \mathrm{Length}(\gamma)\leq 2+\epsilon.
		\end{equation}
		Let $t_0\in[0,1]$ so that $\gamma(t_0)\in \Sigma'$ and $\gamma\notin\Sigma'$ for all $t\in(t_0,1)$. It follows that $\gamma'=\gamma|_{[t_0,1]}$ intersects $\Sigma'$ with algebraic intersection number $1$. Now we consider the minimizing problem of the following functional
		\[  \mc A^1(\Omega'):=\mc H^2(\partial \Omega'\setminus (\Sigma'\cup T))-\mc H^3(\Omega')\]
		among all domains $\Omega'\subset M_2$ that contain $\Sigma'$. Let $\Omega$ be a minimizer of $\mc A^1$. Then by \cite{Mor03}*{Corollary 3.8}, $\partial \Omega\setminus T$ is a smooth, embedded, stable cmc surface because $H_{\partial M}\geq 1$ and $|H_{\Sigma}|\leq 1$. Note that $\partial \Omega$ intersects $\gamma'$ with algebraic intersection number 1. Let $\Gamma$ be a connected component of $\partial \Omega\setminus T$ that intersects $\gamma'$. It follows that $\partial \Gamma\subset T$. 
		
		Now we take $q\in \gamma'\cap \Gamma$. By \eqref{eq:p to T} and \eqref{eq:gamma to p}, together with triangle inequalities,
		\[  \dist_{M_2}(q,T)\geq \dist_{M_2}(p,T\cap M_2)-\dist_M(p,q)> \frac{4\pi}{3}+\epsilon.\]
		Then applying Proposition \ref{lem:bound radius of cmc}, $\partial \Gamma=\emptyset$ since $\Gamma$ is stable. Thus we conclude that $\Gamma$ is a closed embedded stable cmc surface and then it contradicts  $\mathrm{Ric}\geq 0$. Hence, this completes the proof of Theorem \ref{thm:compact boundary}.
		
	\end{proof}

	Observe that \cite{ABCS}*{Lemma 48} gives an upper bound of length of boundary for two-sided, free boundary minimal surfaces with index 1. Moreover, the number of connected components is also uniformly bounded. Together with Theorem \ref{thm:compact boundary}, we obtain a diameter upper bound for free boundary minimal surface with index less than or equal to 1.
	
	\begin{lemma}[c.f. \citelist{\cite{CF20}*{Proposition 1.8} \cite{ABCS}*{Lemma 48}}]\label{lem:stable fbms}
		Let $(M^3,\partial M,g)$ be a complete Riemannian manifold with non-empty smooth boundary and $\Ric(g)\geq 0$, $H_{\partial M}\geq 1$. Let $\Sigma$ be a two-sided, embedded free boundary minimal surface in $M$.
		\begin{enumerate}
			\item\label{item:2-sided} If $\Sigma$ is stable, then $\Sigma$ is a disk and 
			\[ \sup_{x,y\in\Sigma}\dist_\Sigma(x,y)\leq \pi+\frac{8}{3}.\]
			\item\label{item:1-sided} If $\Sigma$ has index one, then
			\[ \sup_{x,y\in \Sigma}\dist_M(x,y)\leq \frac{59\pi}{3}+28.  \]
		\end{enumerate}
	\end{lemma}
	\begin{proof}
		The statement \eqref{item:2-sided} is given by Carlotto-Franz \cite{CF20}*{Proposition 1.8}. Therefore, it suffices to prove statement $(2)$ as follows.
		
		Since $\Sigma$ has index one, then by \cite{ABCS}*{Lemma 48}, 
		\[ |\partial\Sigma|\leq 2(8-r)\pi,\]
		where $r\geq 1$ is the number of the connected components of $\partial \Sigma$. Denote by $C_1,\cdots, C_r$ the connected components of $\partial \Sigma$. Then by Theorem \ref{thm:compact boundary}, for each $C_i$, there exists $C_j\neq C_i$, so that 
		\[\dist_{ M}(C_i,C_j)\leq \frac{8\pi}{3}+4.\] 
		Note that Theorem \ref{thm:compact boundary} gives that for any $x\in \Sigma$,
		\[ \dist_{ M}(x,\partial \Sigma)\leq \frac{4\pi}{3}+2.\]
		{Thus for any $x,y\in \Sigma$,
			\[\dist_{ M}(x,y)\leq r(\frac{8\pi}{3}+4)+\frac{1}{2}|\partial \Sigma|\leq 2r(\frac{4\pi}{3}+2)+(8-r)\pi.  \]
		}
		Since $\partial \Sigma$ is non-empty, we obtain that $r\leq 7$. It follows that
		\[ \dist_{ M}(x,y)\leq \frac{59\pi}{3}+28. \]
		The statement \eqref{item:1-sided} is proved.
	\end{proof}

	\subsection{Geometrically prime regions}
	In this part, we will introduce a class of manifolds obtained from  manifolds by cutting along properly embedded free boundary minimal surfaces, which will be used in the next sections. The new boundary components generated from cutting process are called  {\em portions}. Precisely, we introduce the following definition.
	
	\begin{definition}
		$(N,\partial_rN,T,g)$ is said to be a Riemannian manifold with {\em relative boundary $\partial_rN$ and portion $T$} if 
		\begin{enumerate}[label=(\arabic*)]
			\item  $\partial_rN\cup T$ is exactly the topological boundary of $N$ and $(N,\partial_rN\cup T,g)$ is a Riemannian manifold with piecewise smooth boundary;
			\item  $\partial_rN$ and $T$ are smooth hypersurfaces
			;
			\item $\partial_rN\cap  \mathrm{Int}(T) =\emptyset$ and $\partial_rN$ is transverse to $T$.
		\end{enumerate}
	\end{definition}
	
	Recall that $(\Sigma,\partial \Sigma)\subset(N,\partial_rN)$ always denotes a surface in $N$ with  boundary $\partial \Sigma\subset \partial_rN$. Let $(\Sigma,\partial \Sigma)\subset (N,\partial_rN,T,g)$ be an embedded free boundary minimal surface and $N'$  the metric completion of $N\setminus \Sigma$. Conventionally, we always let 
	\[ \partial_rN'=\partial N\cap N' \ \ \text{ and } \ \ T'=\partial N'\setminus \partial_rN.\]
	Clearly, $(N',\partial_rN',T',g)$ is a Riemannian manifold with relative boundary and portion.
	
	\medskip
	For Riemannian manifolds with relative boundary and portion, we generalize the concept of geometrically prime manifolds given by Liokumovich-Maximo \cite{LM20}*{Definition 2.4}. 
	\begin{definition}\label{def:geo prime}
		Let $(N^3,\partial_rN,T,g)$ be a Riemannian manifold with relative boundary $\partial_rN$ and portion $T$. Denote by $T_0$ the union of  connected components of $T$ that are unstable free boundary minimal surfaces. Then $N$ is said to be {\em geometrically prime} if 
		\begin{enumerate}
			\item $T_0$ is a connected free boundary minimal surface of Morse index 1 unless  $T_0=\emptyset$;
			\item every  closed curve $\gamma$ bounds {\em a surface $\Gamma$ in $N$ relative to $T$}, i.e. $\partial \Gamma\setminus T=\gamma$.
		\end{enumerate}
	\end{definition}

	For geometrically prime Riemannian manifolds with non-empty boundary and portion, we adapt Gromov-Lawson's trick \cite{GLPSC} to bound the diameter of level sets of the distance functions. We state the results here and will defer the details in Appendix \ref{sec:GL trick} for the sake of completeness.
	\begin{proposition}\label{prop:non-empty T_0}
		Let $(N^3,\partial_rN,T,g)$ be a three dimensional geometrically prime Riemannian manifold with non-empty relative boundary $\partial_rN$. Suppose
		\begin{itemize}
			\item $\Ric(g)\geq 0$, $H_{\partial_rN}\geq 1$;
			\item $T_0\neq \emptyset$ and $\mk d$ is the distance function to $T_0$.
		\end{itemize} 
		Then for any continuous curve $\gamma:[0,1]\to N$ with 
		\[\mk d(\gamma(0))=\mk d(\gamma(1))\leq \mk d(\gamma(t)) \text{ \ \  for all \ } t\in[0,1],  \]
		we have
		\[  \dist_N(\gamma(0),\gamma(1))\leq 25\pi+36.\]
		
	\end{proposition}
	
	Moreover, if $T$ is stable, the upper bound can be improved by the following proposition. The proof is parallel to Proposition \ref{prop:non-empty T_0}.
	
	\begin{proposition}\label{prop:distance bounds}
		Let $(N^3,\partial_rN,T,g)$ be a three dimensional geometrically prime Riemannian manifold with non-empty relative boundary $\partial_rN$. Suppose that
		\begin{itemize}
			\item $\Ric(g)\geq 0$, $H_{\partial_r N}\geq 1$;
			\item $T$ is stable;
			\item $p\in N$ is a fixed point and $\mk d$ is a distance function from $p$.
		\end{itemize} 
		Then for any continuous curve $\gamma:[0,1]\to N$ with 
		\[\mk d(\gamma(0))=\mk d(\gamma(1))\leq \mk d(\gamma(t)) \text{ \ \  for all \ } t\in[0,1],  \]
		we have
		\[  \dist_N(\gamma(0),\gamma(1))\leq 8\pi+12.\]
		
	\end{proposition}

	\section{Decomposing manifolds into geometrically prime regions}
	\label{sec:decomp}

	\subsection{Free boundary minimal surfaces with index one}
	In this subsection, we will consider $(N,\partial_rN,T,g)$ a Riemannian manifold with relative boundary and portion that satisfies the following assumptions:
	\begin{enumerate}[label=(\Alph*)]
		\item\label{item:stable T} Each connected component of $T$ is a stable free boundary minimal surface;
		\item\label{item:separate} Each compact, two-sided, properly embedded surface in $(N\setminus T,\partial_rN,g)$ separates $N$;
		\item\label{item:2-sided isotopy} For each two-sided, embedded, stable free boundary minimal surface $\Gamma$, $N\setminus \Gamma$ has a connected component whose metric completion is diffeomorphic to $\Gamma\times [0,1]$;
		\item\label{item:unstable cover} Each one-sided, embedded free boundary minimal surface has an unstable double cover;
		\item\label{item:frankel} Any two one-sided, embedded free boundary minimal surfaces intersect each other.
	\end{enumerate}
	
	\bigskip
	Our aim is to construct a two-sided, index one, free boundary minimal surface that separates the manifolds into two geometrically prime regions.
	
	\bigskip
	
	Let $(N,\partial_rN,T,g)$ be a Riemannian manifold with relative boundary and portion. Denote by $\mc U_\Lambda(N)$ the collection of one-sided, stable free boundary minimal surfaces whose double covers are stable and have area less than or equal to $\Lambda$. 
	
	Now we introduce a {\em $(\mc U,\Lambda)$-process} to remove the elements in $\mc U_\Lambda(N)$:
	Let $p\in N$ be a fixed point. Then we take a sequence of disjoint surfaces $\{\Sigma_j\}\subset \mc U_\Lambda(N)$ satisfying
	\[ \dist_N(\Sigma_j,p)= \inf\{\dist_N(\Sigma',p);\Sigma'\in\mc U_\Lambda(N) \text{ does not intersect $\Sigma_i$ for all $i\leq j-1$}\} .\]
	The existence of $\Sigma_j$ is guaranteed by the compactness of $\mc U_\Lambda(N)$. It suffices to prove that $\dist_N(p,\Sigma_j)\to\infty$ provided that $\{\Sigma_j\}$ has infinitely many elements. Suppose not, $\Sigma_j$ smoothly converges to some $S\in \mc U_\Lambda(N)$ by the compactness of stable free boundary minimal surfaces with bounded area; c.f. \cite{GLZ16}. Since $\Sigma_j$ does not intersect $\Sigma_i$ for $i\neq j$, then $\Sigma_j$ does not intersect $S$. Then in the metric completion of $N\setminus S$, $\Sigma_j$ smoothly converges to the double cover of $S$, which contradicts that $S_j$ are one-sided.

	Denote by $N'$ be the metric completion of $N\setminus \cup_j\Sigma_j$. By convention, $\partial_r N'=\partial_rN\cap N'$ and $T'=\partial  N'\setminus \partial_rN'$. Note that each connected component of $T'\setminus T$ is a double cover of some $\Sigma_j$. Furthermore, $( N',\partial_r N',T',g)$ does not contain any elements in $\mc U_\Lambda(N)$. From the construction, it is also clear that $\{\Sigma_j\}$ contains only finitely many elements provided that $N$ is compact.

	\begin{proposition}\label{prop:cut two side to get prime}
		Let $(N,\partial_rN,T,g)$ be a Riemannian manifold with relative boundary and portion satisfying  Assumptions \ref{item:stable T}--\ref{item:frankel} and $\Ric(N)\geq 0$, $H_{\partial_r N}\geq 1$. If $N$ is not geometrically prime, then there exists a compact, two-sided, embedded free boundary minimal surface $S$ such that
		\begin{enumerate}
			\item $S$ has index 1;
			\item each connected component of the metric completion of $N\setminus S$ is geometrically prime. 
		\end{enumerate}
	\end{proposition}
	
	\begin{proof}
		The proof is divided into five steps.
		
		\medskip
		{\noindent\bf Step I:} {\em Construct $(\hat N,\partial _r\hat N,\hat T,g)$ which does not contain any two-sided, stable free boundary minimal surfaces and satisfies Assumptions \ref{item:stable T}--\ref{item:frankel}.}
		
		\medskip
		Denote by $\{\Gamma_j\}$ the union of the connected components of $T$. For $\Gamma_1$, we define $\mk B_1$ as the collection of two-sided, stable free boundary minimal surfaces that are homologous to $\Gamma_1$. We suppose that $\mc B_1$ is non-empty; otherwise, we skip this step for $\Gamma_1$ and then consider $\Gamma_2$. By Assumption \ref{item:2-sided isotopy}, for each $\Gamma'\in\mk B_1$, the connected component of $N\setminus \Gamma'$ containing $\Gamma_1$ is diffeomorphic to $\Gamma'\times[0,1)$. 
		\begin{claim}\label{claim:compact B1}
			There exists $\hat \Gamma_1\in \mk B_1$ so that 
			\[  \dist_N(\Gamma_1,\hat \Gamma_1)=\sup_{\Gamma'\in\mk B_1}\dist_N(\Gamma_1,\Gamma').\]
		\end{claim}
		\begin{proof}[Proof of Claim \ref{claim:compact B1}]
			Since $N$ is not geometrically prime, then $N$ contains a closed curve that does not bound a surface relative to $T$, which implies $\sup_{\Gamma'\in\mk B_1}\dist_N(\Gamma_1,\Gamma')<\infty$ . Let $\{\Gamma_j'\}\subset \mk B_1$ be a subsequence such that $\dist_N(\Gamma_1, \Gamma'_j)$ converges to $\sup_{\Gamma'\in\mk B_1}\dist_N(\Gamma_1,\Gamma')$. Together with Lemma \ref{lem:stable fbms}, all of $\{\Gamma'_j\}$ are contained in a compact domain $\Omega$. Recall that by \cite{CF20}*{Proposition 1.8}, the length of $\partial \Gamma'_j$ are uniformly bounded. Then by \cite{FrLi}*{Lemma 2.2}, the area of $\Gamma'_j$ is bounded by a constant depending only on $\Omega$. Thus by the compactness theorem \cite{GLZ16}, a subsequence of $\{\Gamma'_j\}$ smoothly converges to a stable free boundary minimal surface $\hat \Gamma_1$ which is either two-sided, or one-sided with a stable double cover. By Assumption \ref{item:unstable cover}, $\hat\Gamma_1$ is two-sided. This completes the proof of Claim \ref{claim:compact B1}.
		\end{proof}
		
		Denote by $N_1$ the connected component of the metric completion of $N\setminus \hat \Gamma_1$ that does not contain $\Gamma_1$. Clearly, $N_1$ is diffeomorphic to $N$ and there is no two-sided, stable free boundary minimal surface in $(N_1,\partial_rN\cap N_1,\hat \Gamma_1\cup T\setminus \Gamma_1,g)$ that is isotopic to $\hat \Gamma_1$. By the same argument for each $\Gamma_j$, we obtain a region $\hat N\subset N$ such that $(\hat N,\partial_r\hat N,\hat T )$ satisfies  Assumptions \ref{item:stable T}--\ref{item:frankel} and $\hat N\setminus \hat T$ does not contain any two-sided, stable free boundary minimal surfaces.
		
		\medskip
		{\noindent\bf Step II:} {\em We approximate $\hat N$ by compact domains that have no one-sided, stable free boundary minimal surfaces with small area.}
		
		\medskip
		Let $\gamma\subset N$ be a closed curve that does not bound a surface relative to $T$. By the construction of $\hat N$, there exists $\hat \gamma\subset \hat N$ which is isotopic to $\gamma$ in $N$. It follows that $\hat \gamma$ does not bound a surface in $\hat N$ relative to $\hat T$.
		
		Then we assume that $\{B_k\}$ is an exhausting sequence of compact domains such that $\partial B_k\setminus \partial \hat N$ is smooth and transverse to $\partial \hat N$. Without loss of generality, we assume that $\partial B_k\setminus \partial \hat N$ does not intersect $\hat T$. We choose a metric $g_k$ on $B_k$ such that
		\begin{itemize}
			\item $\partial B_k\setminus \partial \hat N$ is a stable free boundary minimal surface with respect to $g_k$;
			\item $g_k=g$ except in a $1/k$ neighborhood of $\partial B_k\setminus \partial \hat N$ that does not intersect $\hat T$.
		\end{itemize}
		By our convention, let $\partial_rB_k=\partial_r \hat N\cap B_k$ and $T_{B_k}=\partial B_k\setminus \partial_rB_k$. Hence $(B_k,\partial_r  B_k, T_{B_k},g_k)$ is a compact Riemannian manifold with relative boundary and portion. 
		\begin{claim}\label{claim:gamma nontrivial in Bk}
			$\hat \gamma$ does not bound a surface in $B_k$ relative to $T_{B_k}$.
		\end{claim}
		\begin{proof}[Proof of Claim \ref{claim:gamma nontrivial in Bk}]
			Suppose not, then there exists a surface $\Gamma'$ with $\partial \Gamma'\setminus T_{B_k}=\hat \gamma$. Recall that each connected component of $\hat T$ is a disk. Thus we can take $\Gamma'$ satisfying $\partial \Gamma'\setminus (T_{B_k}\setminus \hat T)=\gamma$. Then there exists a minimizing surface $F\subset B_k$ among all of these $\Gamma'$ described as above. Recall that $\dist_N(\hat \gamma,T_{B_k}\setminus \hat T)$ is sufficiently large. By Theorem \ref{thm:compact boundary}, $F$ is a compact minimal surface and does not intersect $\{g\neq g_k\}$. This contradicts that $\hat \gamma$ does not bound a surface in $\hat N$ relative to $\hat T$.
		\end{proof}
		
		By Claim \ref{claim:gamma nontrivial in Bk}, for some fixed large $k_0$, there is a minimizing surface $\Sigma$ in $(B_{k_0},\partial_rB_{k_0},g_k)$ intersecting $\hat \gamma$ with algebraic intersection number 1. Note that $\Sigma$ is also an embedded surface in $(\hat N,\partial_r N,T,g)$.

		Now we always take $k\geq k_0$. Applying $(\mc U,2\mr{Area}(\Sigma,g))$-process to $(B_k,\partial_rB_k,T_{B_k},g_k)$, we obtain finitely many disjoint surfaces $\{G_k^j\}_j$ such that $(\hat B_k,\partial_r\hat B_k, T_{\hat B_k},g_k)$ does not contain any one-sided stable free boundary minimal surfaces whose double covers are stable and have area less than or equal to $2\mr{Area}(\Sigma,g)$. Here $\hat B_k$ is the metric completion of $B_k\setminus \cup_jG_k^j$ and 
		$ \partial_r\hat B_k=\hat B_k\cap \partial_r B_k, \  T_{\hat B_k}=\partial \hat B_k\setminus \partial_r\hat B_k$. By  next claim, we conclude that $\hat B_k$ converges to $\hat N$ in the sense of Gromov-Hausdorff. 
		\begin{claim}\label{claim:G to infty}
			$\inf_{j}\dist_{g_k}(G_k^j,\hat \gamma)\to\infty$ as $k\to\infty$.
		\end{claim}
		\begin{proof}[Proof of Claim \ref{claim:G to infty}]
			Assume that $G_k^1$ achieves $\inf_{j}\dist_{g_k}(G_k^j,\hat \gamma)$. Suppose on the contrary that $\dist_{g_k}(G_k^1,\hat \gamma)$ remains uniformly bounded. Then by the compactness for stable free boundary minimal surfaces \cite{GLZ16}, $G_k^1$ smoothly converges to a one-sided stable free boundary minimal surface $G\subset \hat N$ whose double cover is stable. This contradicts the construction of $\hat N$. Therefore, we conclude that $\dist_{g_k}(G^1_k,\hat \gamma)\to\infty$ as $k\to\infty$.
		\end{proof}
		
		\medskip
		{\noindent\bf Step III: } {\em Cut along two-sided stable free boundary minimal surfaces with small area in compact domains with perturbed metrics.}
		
		\medskip
		Let $\mk F_k$ be the collection of two-sided, stable free boundary minimal surfaces 
		\[(\Gamma',\partial \Gamma')\subset ( \hat B_k,\partial_r\hat B_k,g_k) \]
		with $\Gamma'\subset\hat B_k\setminus T_{\hat B_k}$ and $\mr{Area}(\Gamma', g_k)\leq 2\mr{Area}(\Sigma,g)$. Now we assume that $\mk F_k$ is non-empty; otherwise, we skip this step for $\hat B_k$. 
		\begin{claim}\label{claim:existence of Sigma1}
			There exists $\Sigma_k^1\in\mk F_k$ so that
			\[ \dist_{N}(\Sigma_k^1,\hat\gamma)=\inf_{\Sigma'\in \mk F_k}\dist_{N}(\Sigma',\hat\gamma).\]
		\end{claim}
		\begin{proof}[Proof of Claim \ref{claim:existence of Sigma1}]
			Suppose that $\{\Sigma'_j\}\subset \mk F_k$ and $\dist(\Sigma'_j, \hat{\gamma})$ converges to $\inf_{\Sigma'\in \mk F_k}\dist_{N}(\Sigma',\hat\gamma)$. Then by the compactness theorem \cite{GLZ16}, $\Sigma_j'$ smoothly converges to a stable free boundary minimal surface $\Sigma_k^1$ in $(\hat B_k,\partial_r\hat B_k)$. Moreover, $\Sigma_k^1$  either is two-sided or have a stable double cover. Note that $(\hat B_k,\partial_r\hat B_k)$ does not contain a stable free boundary minimal surface whose double cover is stable and has area less than or equal to $2\mr{Area}(\Sigma,g)$. Thus we conclude that $\Sigma_k^1$ is two-sided. This finishes the proof of Claim \ref{claim:existence of Sigma1}.
		\end{proof}
		
		By a similar argument as in Claim \ref{claim:G to infty}, we also have
		$\dist_{N}(\Sigma_k^1,\hat \gamma)\to\infty$ as $k\to\infty$.
		Then for any large fixed $k$, let $\hat N_k^1$ denote the connected component of the metric completion of $\hat B_k\setminus \Sigma^1_k$ that contains $\hat\gamma$. If $\hat N_k^1\setminus \Sigma_k^1$ contains elements in $\mk F_k$, then we take $\Sigma_k^2\in \mk F_k$ that is contained in $\hat N_k^2\setminus \Sigma_k^1$ so that 
		\[ \dist_{N}(\Sigma_k^2,\hat\gamma)=\inf\big\{\dist_{N}(\Sigma',\hat\gamma);\Sigma'\in \mk F_k \text{ is contained in }  \hat N_k^2 \setminus \Sigma_k^1\big\}.\]
		By continuing this argument, we obtain two sequences $\{\Sigma_k^j\}_j$ and $\{\hat N_k^j\}_j$. 
		
		Now we are going to prove that this sequence $\{\Sigma_k^j\}_j$ consists of finitely many surfaces. Suppose not, By the compactness theorem \cite{GLZ16}, we have $\Sigma_k^j$ smoothly converges to a stable free boundary minimal surface $C_k$, which  either is two-sided or has a stable double cover. Recall that $\hat B_k\setminus T_{\hat B_k}$ does not contain an embedded one-sided free boundary minimal surface with stable double cover. Hence $C_k$ is two-sided. Then the stability of $C_k$ gives that $\Sigma_k^j$ lies in one side of $C_k$ for all large $j$. From the construction of $\hat N_k^j$, $\hat \gamma$ and $C_k$ lie in the same side of $\Sigma_k^j$. It follows that
		\[\dist_N(\Sigma_k^j,\hat \gamma)<\dist_N(C_k,\hat\gamma)\]
		for all sufficiently large $k$. This contradicts the choice of $\Sigma_k^j$. Thus $\{\Sigma_k^j\}_j$ is finite.

		For simplicity, let $\Sigma_k:=\cup_j\Sigma_k^j$ and $\wti N_k$ be the connected component of the metric completion of $\hat B_k\setminus \Sigma_k$ that contains $\hat \gamma$. Note that 
		\[  \partial_r\wti N_k=\partial_r\hat N\cap \wti N_k \ \ \text{ and } \ \ \wti T_k=\partial\wti N_k\setminus \partial_r\wti N_k.\]
		By Claim \ref{claim:G to infty} and the fact that $\dist_N(\Sigma_k,\hat\gamma)\to \infty$, we have that $(\wti N_k,g_k)$ converges to $(\hat N,g)$ in the sense of Gromov-Hausdorff convergence. Hence $\wti N_k$ contains $\Sigma$ for all large $k$. Moreover, by the construction of $\Sigma_k$, $(\wti  N_k\setminus \wti T_k,\partial_r\wti N_k,g_k)$ does not contain any stable free boundary minimal surfaces that are
		\begin{itemize}
			\item two-sided with area less than or equal to $2\mr{Area}(\Sigma,g)$;
			\item one-sided and has stable double cover with area less than or equal to $2\mr{Area}(\Sigma,g)$.
		\end{itemize} 
		Furthermore, $\wti N_k$ satisfies a {\em $2\mr{Area}(\Sigma,g)$-Frankel property}: any two free boundary minimal surfaces intersect if they are two-sided (or one-sided) and have area less than or equal to $\mr{Area}(\Sigma,g)$ (resp. $\mr{Area}(\Sigma,g)$).

		\medskip
		{\noindent\bf Step IV: } {\em The first width of $(\wti N_k,g_k)$ is bounded by $2\mr{Area}(\Sigma,g)$.}
		
		\medskip
		Recall that $\wti N_k$ contains $\Sigma$ which intersects $\hat \gamma$ with algebraic intersection number 1. By a minimizing process, there exists a free boundary minimal surface $E_k\subset (\wti N_k\setminus \wti T_k,\partial_r\wti N_k,g_k)$ with 
		\[\mr{Area}(E_k,g_k)\leq \mr{Area}(\Sigma,g_k)=\mr{Area}(\Sigma,g).\]
		By Steps II and III, $E_k$ is one-sided and has unstable double cover. Then by a similar argument as in \cite{Wang20}*{Lemma 2.5}, there exists a neighborhood $\mc N_k\subset \wti N_k$ of $E_k$ that is foliated by free boundary surfaces with mean curvature vector pointing away from $E_k$. Let $(\Omega_k^t)_t$ denote the free boundary level set flow starting from $\wti N_k\setminus \mc N_k$. Then by \cite{EHIZ}*{Theorem 1.5}, each connected boundary of $\partial \Omega_k^\infty$ is 
		\begin{itemize}
			\item either a two-sided stable free boundary minimal surface with area less than or equal to $2\mr{Area}(\Sigma,g)$;
			\item or a one-sided stable free boundary minimal surface with area less than or equal to $\mr{Area}(\Sigma,g)$ and its double cover is stable.
		\end{itemize}
		By the construction of $\wti N_k$, there are no such surfaces in $\wti N_k\setminus \wti T_k$. Thus we conclude that $\partial\Omega_k^\infty=\wti T_k$. As a corollary, each connected component of $\wti T_k$ has {\em a contracting neighborhood}, i.e. for each connected component $\Gamma'$ of $T_k$, there exists a neighborhood of $\Gamma'$ in $\wti N_k$ that is foliated by free boundary surfaces with mean curvature vector pointing towards $\Gamma'$. Moreover, the above argument gives that the first width for $\wti N_k$ is bounded from above by $2\mr{Area}(\Sigma,g)$.

		\medskip
		{\noindent\bf Step V: } {\em Apply min-max theory to $N_k$ to find the desired surfaces.}
		
		\medskip
		By min-max theory \cite{LZ16} (see \cite{Wang20}*{Theorem 3.10} for manifolds with relative boundaries and portions), there exist an integer $m$ and an embedded free boundary minimal surface $(E_k,\partial E_k)$ in $(\wti N_k\setminus \wti T_k,\partial_r\wti N_k,g_k)$ such that 
		\[ m\mr{Area}(E_k;g_k)\leq 2\mr{Area}(\Sigma;g), \ \ \text{ and } \ \ \mr{Index}(E_k)\leq 1.\]
		Moreover, by Catenoid Estimates in \cite{KMN16} (see \cite{Wang} for a free boundary version), $E_k$ is two-sided. Recall that $(\hat N_k\setminus\hat T_k,\partial_r\hat N_k,g_k)$ does not contain any two-sided stable free boundary minimal surfaces with area less than or equal to $2\mr{Area}(\Sigma)$. Thus $E_k$ has index one. Then by the $2\mr{Area}(\Sigma,g)$-Frankel property, $E_k$ intersects $\Sigma\cup\hat \gamma$ for all large $k$; otherwise, one can construct a one-sided free boundary minimal surface that does not intersect $E_k$ and has an unstable double cover. Letting $k\to\infty$, one can obtain a free boundary minimal surface in $\hat N$ which is 
		\begin{itemize}
			\item either two-sided and has index one;
			\item or one-sided and has a double cover with index less than or equal to one.
		\end{itemize}
		In both cases, the limit of $E_k$ is compact by Lemma \ref{lem:stable fbms}. Thus $E_k$ is a free boundary minimal surface in $(\hat N,\partial_r\hat N,\hat T,g)$ for all sufficiently large $k$.
		
		Now we pick $S=E_k$ and then prove that $S$ is the desired surface. Let $W_1$ and $W_2$ be two connected components of the metric completion of $N\setminus S$. Let $\hat W_1$ and $\hat W_2$ be two connected components of the metric completion of $\hat N\setminus S$ respectively.
		
		\begin{claim}\label{claim:prime}
			$W_1$ and $W_2$ are both geometrically prime.
		\end{claim}
		\begin{proof}[Proof of Claim \ref{claim:prime}]
			Suppose on the contrary that $W_1$ is not geometrically prime. Then there exists a one-sided, compact, connected stable free boundary minimal surface $V\subset W_1$. Recall that $\hat N$ is obtained from by cutting countably many domains which are diffeomorphic to $\Gamma'\times[0,1]$ for some disk $\Gamma'$. Thus we can take $V\subset \hat W_1$.
			
			{By Assumption \ref{item:unstable cover}}, the double cover of $V$ is unstable. Let $\wti W_1$ be the metric completion of $\hat W_1\setminus V$. Then $\wti W_1$ contains two disjoint unstable free boundary minimal surfaces: one is $S$ and the other one is the double cover of $V$. Thus by a minimizing process, $\wti W_1\setminus T$ contains a two-sided stable free boundary minimal surface $S'$, which contradicts the construction of $\hat N$. This completes the proof of Claim \ref{claim:prime}.
		\end{proof}
		Therefore, $S$ is the desired free boundary minimal surface and Proposition \ref{prop:cut two side to get prime} is proved.
	\end{proof}

	\subsection{Geometrically prime decomposition}
	In this subsection, we will decompose a (possibly non-compact) Riemannian manifold with boundary into geometrically prime regions.

	Let $\mc O_S$ (resp. $\mc U_S$) be the collection of two-sided (resp. one-sided) stable free boundary minimal surfaces. Let $\mc U_S^1$ (resp. $\mc U_S^2$) be the collection of $\Sigma\in\mc U_S$ whose double cover is stable (resp. unstable). 
	
	Observe that a sequence of surfaces in $\mc O_S$ or $\mc U_S^1$ converges subsequently if they are bounded. Indeed, Lemma \ref{lem:stable fbms} gives an upper bound of the length of their boundaries. Together with \cite{FrLi}*{Lemma 2.2}, their areas are uniformly bounded. Then the convergence of subsequences would follow from the compactness theorem in \cite{GLZ16}.

	\begin{lemma}\label{lem:cut}
		Let $(M,\partial M,g)$ be a three dimensional Riemannian manifold with non-empty boundary. Suppose that $M$ satisfies that $\Ric(g)\geq 0$ and $H_{\partial M}\geq 1$. Then there exist countably many disjoint free boundary minimal surfaces $\{P_j\}$, $\{D_j\}$ and $\{S_j\}$ such that
		\begin{enumerate}
			\item $\{D_j\}\subset \mc O_S$, $\{P_j\}\subset \mc U_S^1$  and $\{S_j\}$ are two-sided free boundary minimal surfaces of index $1$;
			\item Each connected component of the metric completion of  $M\setminus \mc D$ is geometrically prime. Here 
			\[\mc D=\bigcup_{i,j,k}(P_i\cup D_j\cup S_k).\]
		\end{enumerate}
	\end{lemma}
	\begin{proof}
		Let $p\in M$ is a fixed point.
		
		\medskip
		\noindent
		{\bf Step I: } {\em There exists a sequence of disjoint surfaces $\{P_j\}\subset \mc U_S^1$ such that every $\Gamma' \in \mc U_S^1$ intersects $P_j$ for some $j\geq 1$. }
		
		\medskip
		We will choose these surfaces inductively. Suppose we have chosen $\{P_j\}_{j\leq k}$. Then we take $P_{k+1}$ that minimizes $\dist_M(p,\Gamma')$ among all $\Gamma'$ satisfying the following:
		\begin{itemize}
			\item $\Gamma'\in \mc U_S^1$;
			\item $\Gamma'$ does not intersect $P_j$ for all $1\leq j\leq k$.
		\end{itemize}
		To finish Step I, it suffices to prove that $\dist_M(p,P_j)\to \infty$ provided that $\{P_j\}$ is an infinite set. Suppose not, then by the compactness theorem in \cite{GLZ16}, $P_j$ smoothly converges to an element $P\in \mc U_S^1$. Since $P_j$ does not intersect $P_i$ for $i\neq j$, then $P_j$ does not intersect $P$. It follows that $P_j$ smoothly converges to the double cover of $P$. Hence $P_j$ is two-sided, it contradicts the choice of $P_j$. This finishes the proof of Step I.
		
		\medskip
		Let $M_1$ be the metric completion of $M\setminus \cup_jP_j$. Then by Step I, there is no surface $\Gamma'\subset M_1$ that belongs to $\mc U_S^1$.
		
		\medskip
		{\noindent\bf Step II:} {\em There exists a sequence of disjoint surfaces $\{C_j\}\subset \mc O_S$ in $M_1$ such that the metric completion of $M_1\setminus\cup_jC_j$ satisfies \ref{item:separate}.}
		
		\medskip
		We use the inductive method again. Suppose we have chosen $\{C_j\}_{j\leq k}$. Then we take $C_{k+1}$ that minimizes $\dist_M(\Gamma',p)$ among all $\Gamma'$ satisfying the following:
		\begin{itemize}
			\item $\Gamma'\in \mc O_S$;
			\item $\Gamma'\subset M_1\setminus \cup_{j\leq k}C_j$;
			\item $\Gamma'$ does not separate $M_1\setminus \cup_{j\leq k}C_j$.
		\end{itemize}
		We now prove that $\dist_M(p,C_j)\to \infty$ provided that $\{C_j\}$ consists of infinitely many elements. Suppose not, then by \cite{GLZ16}, $C_j$ smoothly converges to a stable free boundary minimal surface $C$ that belongs to $\mc O_S$ or $\mc U_S^1$. By Step I, $(M_1,M_1\cap \partial M)$ does not contain surfaces in $\mc U_S^1$. Thus $C\in \mc O_S$. Then by the smooth convergence, $C_k$ is a positive graph over $C$ for all sufficiently large $k$, which contradicts that $C_{k+1}$ does not separate $M_1\setminus \cup_{j\leq k}C_j$. 
		
		\medskip
		Denote by $M_2$ the metric completion of $M_1\setminus \cup_jC_j$. Then each compact, two-sided, embedded surface in $M_2$ with boundary on $M_2\cap \partial M$ separates $M_2$. Otherwise, by a minimizing process, one can find an $S\in \mc O_S$ that does not separate $M_2$ and $\dist_M(p,S)<\infty$, which leads to a contradiction.
		
		\medskip
		{\noindent\bf Step III:} {\em There exists a sequence of disjoint surfaces $\{E_j\}\subset \mc O_S$ such that the metric completion of $M_2\setminus \cup_j E_j$ satisfies \ref{item:2-sided isotopy}.}
		
		\medskip
		We construct these surfaces inductively. Suppose we have chosen $\{E_j\}_{j\leq k}$. Then we take $E_{k+1}$ that minimizes $\dist_M(\Gamma',p)$ among all $\Gamma'$ satisfying the following:
		\begin{itemize}
			\item $\Gamma'\in\mc O_S$;
			\item $\Gamma'\subset M_2\setminus \cup_{j\leq k}E_j$;
			\item the metric completion of $M_2\setminus \cup_{j\leq k+1}E_j$ does not have a connected component whose closure is diffeomorphic to $\Gamma'\times [0,1]$.
		\end{itemize}
		Now we are going to prove $\dist_M(p,E_j)\to \infty$ provided that $\{E_j\}$ consists of infinitely many elements. Suppose not, $E_j$ smoothly converges to an embedded surface $S\in \mc O_S$ by the same argument as that in Step II. Then by the smooth convergence, $E_j$ lies on one side of $S$ for all sufficiently large $j$. Then the metric completion of the connected component of $M_2\setminus \cup_{i\leq j+1}E_i$ that contains $E_j$ and $E_{j+1}$, is diffeomorphic to $E_j\times[0,1]$ that contradicts the choice of $\{E_j\}$.
		
		\medskip
		Now let $M_3$ be the metric completion of $M_2\setminus \cup_j E_j$. Clearly, $M_3$ satisfies \ref{item:2-sided isotopy}.

		\medskip
		{\noindent\bf Step IV:} {\em There exists a sequence of two-sided, index one, free boundary minimal surfaces $\{S_j\} \subset M_3$ such that each connected component of the metric completion of $M_3\setminus\cup_jS_j$ is geometrically prime.}
		
		\medskip
		Let $N$ be a connected component of $M_3$ and 
		\[ \partial_rN=\partial M\cap N\ \ \text{ and } \ \ T=\partial N\setminus \partial_rN.\]
		Now we verify that $(N,\partial_rN,T)$ satisfies Assumptions \ref{item:stable T}--\ref{item:frankel}. Note that every connected component of $T$ is from one of the following:
		\begin{itemize}
			\item a double cover of $P_j\in \mc U_S^1$;
			\item one of $C_j\in \mc O_S$;
			\item one of $E_j\in \mc O_S$.
		\end{itemize}
		Thus $N$ satisfies \ref{item:stable T}. By Step I and II, \ref{item:unstable cover} and \ref{item:separate} are satisfied respectively and Step III gives \ref{item:2-sided isotopy} immediately.

		Finally, it remains to verify \ref{item:frankel}. Suppose not, then there exist two disjoint surfaces $S_1,S_2\in \mc U_S^2$ that are contained in $N$. Let $\wti N$ be the metric completion of $N\setminus (S_1\cup S_2)$. Note that $S_1$ and $S_2$ have unstable double covers since $N$ satisfies \ref{item:unstable cover}. Let $\wti S_1$ and $\wti S_2$ be the unstable free boundary minimal surface in $\wti N$ arising from cutting along $S_1$ and $S_2$ respectively. By taking an area minimizer of the homology class in $H_2(\wti N,\partial_r\wti N;\mb Z)$ represented by $\wti S_1$, we obtain a stable free boundary minimal surface $\Gamma\in \mc O_S$. Since $N$ satisfies \ref{item:separate}, then $\Gamma$ separates $N$. Moreover, $S_1$ and $S_2$ lie in two different connected components of $N\setminus \Gamma$, which contradicts the choice of $\{E_j\}$ in Step III. Hence \ref{item:frankel} is satisfied.
		
		\medskip
		Let $\{N_j\}$ be the collection of connected components of $M_3$. Thus each $N_j$ satisfies Assumptions \ref{item:stable T}--\ref{item:frankel}. By Proposition \ref{prop:cut two side to get prime}, there exists a two-sided free boundary minimal surface $S_j$ with index 1 such that the metric completion of $N_j\setminus S_j$ is geometrically prime provided that $N_j$ is not geometrically prime. This finishes Step IV.
		
		\medskip
		Therefore, Lemma \ref{lem:cut} follows by relabelling $\{C_j\}\cup\{E_j\}$ as $\{D_j\}$.

	\end{proof}

	\section{Upper bounds for Uryson 1-width}\label{sec:bound}
	In this section, we are in a position to prove an upper bound of Uryson 1-width for all three dimensional Riemannian manifolds with non-negative Ricci curvature and strictly mean convex boundary. By the definition, it suffices to construct a continuous function such that every connected component of all the level sets has a uniformly bounded diameter.

	In the first part of this section, we construct a function on each geometrically prime region. By Lawson-Gromov's tricks, the distance function to the unstable component of the portion is a good choice. However, for later reason, the desired function on each connected component in the portion is required to have the same value. Hence we modify the distance function near the portion so that the diameter bound still holds.

	Recall that $T_0$ always denotes the unstable component if $T$ is a free boundary minimal surface by Definition \ref{def:geo prime}.
	\begin{lemma}\label{lem:width for prime}
		Let $(N^3,\partial_r N,T,g)$ be a three dimensional geometrically prime Riemannian manifold with non-empty relative boundary $\partial_r N$ and portion $T$. Suppose that $\Ric(g)\geq 0$, $H_{\partial_rN}\geq 1$ and $T$ is a free boundary minimal surface. Then there exists a continuous function $f:N\to [0,\infty)$ satisfying
		\begin{enumerate}  
			\item $f(x)=0$ for all $x\in T_0$, where $T_0$ is the union of unstable component of $T$;
			\item  $f(x)=1$ for all $x\in T\setminus T_0$;
			\item $\dist_N(x,y)<117$ for any $t$ and $x, y$ in the same connected component of $f^{-1}(t)$.
		\end{enumerate}
		Moreover, the upper bound in the third statement can be improved to be $49$ if $T$ is stable.
	\end{lemma}
	\begin{proof}
		Let $\{T_j\}_{j \geq 1}$ be the collection of the connected components of $T\setminus T_0$. Since all $T_jl$ are compact, then there exist a positive constant $\epsilon<100^{-1}$ and $p\in N\setminus T$ so that 
		\[\dist_N(T\setminus T_0,T_0)>5\epsilon \text{\  if \ } T_0\neq \emptyset \ ; \  \dist_N(T,p)>5\epsilon \text{\  if \ } T_0=\emptyset.
		\]
		and for each $T_j$, there exists $\epsilon_j\in (0,\epsilon/2^j)$ such that 
		$\dist_N(T_j,T\setminus T_j)>5\epsilon_j$. For $j \geq 1$, set
		\[  U_j:=\{x\in N;\dist_N(x,T_j)\leq 2\epsilon_j\};\]
		and for $j=0$, set
		\[ U_0:=\{x\in N;\dist_N(x,T_0)\leq \epsilon\}\ \  \text{ if } \ T_0\neq \emptyset \ ; \ U_0:=\{x\in N;\dist_N(x,p)\leq \epsilon\} \ \ \text{ if } \ T_0=\emptyset.\]
		Let $\eta:[0,\infty)\to[0,1]$ be a continuous cut-off function such that
		\[ \eta(t)=0 \ \ \text{ for } \ \ t\geq 2; \  \ \eta(t)=1 \ \ \text{ for } \ \ t\in[0,1].\]
		Define $h:N\to [0,\infty)$ by
		\[ h(x)=\sum_{j\geq1}\eta(\epsilon_j^{-1}\dist_N(x,T_j)).\]
		It follows that $\eta(\epsilon_j^{-1}\dist_N(x,T_j))=0$ outside $U_j$. Hence $h$ is well-defined, supported in $\cup_jU_j$ and 
		\[ 0\leq h\leq 1; \ \ \  h(x)=0 \text{ \  for \ } x\in N\setminus \cup_{j\geq 1}U_j.\]

		\medskip
		Now we define the desired function. 
		\begin{itemize}
			\item If $T_0\neq \emptyset$, $f:N\to [0,\infty)$ is defined by
			\[ f(x):=h(x)+\epsilon^{-1}(1-h(x))\dist_N(x,T_0).\]
			
			\item 
			If $T_0=\emptyset$, then we take $p\in N$ with $\dist_N(p,T)>5\epsilon$ and define $f$ by
			\[ f(x):=h(x)+\epsilon^{-1}(1-h(x))\dist_N(x,p).\]
		\end{itemize}
		
		For $x\in U_0$, it follows that $\dist_N(x,T_j)\geq 5\epsilon_j$ for $j\geq 1$ and then $h(x)=0$. Hence for $x\in U_0$,
		\begin{equation}\label{eq:f in U_0}
		f(x)=\epsilon^{-1}\dist_N(x,T_0) \text{ \ \ for \ \ } T_0\neq \emptyset; \ f(x)=\epsilon^{-1}\dist_N(x,p)\ \text{  \ \ for \ \  } T_0=\emptyset.
		\end{equation}
		Thus the first statement is satisfied. Also, for all $x\in T\setminus T_0$, we have 
		$h(x)=1$, it implies that $f(x)=1$. This gives the second statement.
		
		\medskip
		Now let's verify that $f$ satisfies the third requirement. For $t\geq0$ and any two points $y,z$ in the same connected component of $f^{-1}(t)$, there exists a continuous curve $\gamma:[0,1]\to N$ with $f(\gamma(s))=t$ for all $s\in [0,1]$ and $\gamma(0)=y$, $\gamma(1)=z$.
		
		We now consider the case that $T_0\neq \emptyset$. If $\dist_N(\gamma(s'),T_0)< \epsilon$ for some $s'\in[0,1]$, then by \eqref{eq:f in U_0}, we conclude that $y,z\in U_0$ and  \[ \dist_N(y,z)\leq \dist_N(y,T_0)+\dist_N(z,T_0)+\sup_{x_1,x_2\in T_0}\dist_N(x_1,x_2)\leq \frac{59\pi}{3}+28+2\epsilon+2\epsilon<117. \]
		
		If $\dist_N(\gamma,T_0)\geq \epsilon$, it follows that 
		\begin{equation}\label{eq:f and d}
		f(\gamma(s))\leq \epsilon^{-1}\dist_N(x,T_0).
		\end{equation}
		Then we have the following two cases.
		
		\medskip
		{\noindent\em Case 1: $h(\gamma(s))>0$ for all $s\in [0,1]$. }
		
		\medskip
		Note that $h$ is supported on disjoint compact sets $\cup_jU_j$. Thus $\gamma\subset U_j$ for some $j\geq 1$. It follows that 
		\[ \dist_N(z,y)\leq \sup_{x_1,x_2\in T_j}\dist_N(x_1,x_2)+2\epsilon_j+2\epsilon_j\leq \pi+\frac{8}{3}+2\epsilon,\]
		which is the desired inequality.
		
		\medskip
		{\noindent\em Case 2: $h(\gamma(s_0)=0$ for some $s_0\in[0,1]$.}
		
		\medskip
		Let 
		\[
		s_1:=\inf\{s\in[0,s_0];h(\gamma(s))=0\}; \ \ s_2:=\sup\{s\in [s_0,1];h(\gamma(s))=0\}.
		\]
		Then $\gamma|_{[0,s_0]}$ (resp. $\gamma|_{[s_1,1]}$) lies in $U_j$ for some $j\geq 1$. Moreover, by the same argument in Case 1,
		\[ \dist_N(y,\gamma(s_1))\leq \pi+\frac{8}{3}+4\epsilon\ \ \text{ and } \ \ \dist_N(z,\gamma(s_2))\leq \pi+\frac{8}{3}+4\epsilon.\]
		Next, we bound the distance from $\gamma(s_1)$ to $\gamma(s_2)$. Indeed, by \eqref{eq:f and d}, for any $s\in[s_1,s_2]$ and $j\in\{1,2\}$,
		\[ \dist_N(\gamma(s),T_0)\geq \epsilon\cdot f(\gamma(s))=\epsilon\cdot f(\gamma(s_j))=\dist_N(\gamma(s_j),T_0).\]
		Then by Proposition \ref{prop:non-empty T_0},
		\[
		\dist_N(\gamma(s_1),\gamma(s_2))\leq 25\pi+36.
		\]
		Then by the triangle inequality,
		\begin{align*} 
		\dist_N(y,z)&\leq \dist_N(y,\gamma(s_1))+\dist_N(\gamma(s_1),\gamma(s_2))+\dist_N(\gamma(s_2),z)\\
		&\leq 27\pi+26+\frac{16}{3}+4\epsilon<117 .
		\end{align*}
		This finishes the proof for $T_0\neq \emptyset$.

		Now it remains to improve the upper bound as $T_0=\emptyset$. If  $\dist_N(\gamma(s'),p)< \epsilon$ for some $s'\in[0,1]$, the inequality is trivial. If $\dist_N(\gamma,T_0)\geq \epsilon$, we only consider the {\em Case 2: $h(\gamma(s_0)=0$ for some $s_0\in[0,1]$.} Then the triangle inequality gives
		\begin{align*}
		\dist_N(y,z)&\leq \dist_N(y,\gamma(s_1))+\dist_N(\gamma(s_1),\gamma(s_2))+\dist_N(\gamma(s_2),z)\\
		&\leq 10\pi+12+\frac{16}{3}+4\epsilon<49.
		\end{align*}
		This finishes the proof of Lemma \ref{lem:width for prime}.
	\end{proof}

	\bigskip
	\begin{theorem}\label{thm:f from M}
		Let $(M^3,\partial M,g)$ be a three dimensional Riemannian manifold with non-empty boundary $\partial M$. Suppose that $\Ric(g)\geq 0$ and $H_{\partial M}\geq 1$. Then there exists a continuous function $f:M\to [0,\infty)$ satisfying that
		\[\dist_M(x,y)\leq 117\]
		for all $t\geq 0$ and $x,y$ in the same connected component of $f^{-1}(t)$. In particular, if $M$ is a domain in $\mb R^3$, then the upper bound can be improved to be $49$.
	\end{theorem}

	\begin{proof}
		By Lemma \ref{lem:cut}, there exists a sequence of free boundary minimal surfaces $\{D_j\}$ such that each connected component of the metric completion of $M\setminus \cup_jD_j$ is geometrically prime.
		
		Let $\{N_j\}_{j\geq 1}$ denote the union of the connected components of the metric completion of $M\setminus \cup_j D_j$. Note that $\partial_rN_j=N_j\cap \partial M$ and $T_j=\partial N_j\setminus \partial_r N_j$. Denote by $T_{j,0}$ the unstable component of $T_j$. Let $\{T_{j,i}\}_{i\geq 1}$ denote the union of the connected components of $T_j\setminus T_{j,0}$. Then by Lemma \ref{lem:width for prime}, for any $N_j$, there exists a continuous function $f_j:N_j\to [0,\infty)$ such that 
		\begin{itemize}  
			\item $f_j(x)=0$ for all $x\in T_{j,0}$, where $T_{j,0}$ is the union of unstable component of $T_{j}$;
			\item  $f_j(x)=1$ for all $x\in T\setminus T_{j,0}$;
			\item $\dist_N(x,y)<117$ if $x$ and $y$ lie in the same connected component of $f_j^{-1}(t)$.
		\end{itemize}
		Since $f_j|_{T_{j,0}}=0$ if $T_{j,0}\neq \emptyset$, and $f|_{T_j\setminus T_{j,0}}=1$, then the gluing of all these functions induces a continuous  function $ f$ on $M$. 
		
		It remains to prove the diameter bound. Let $K$ be a connected component of $ f^{-1}(t)$ for any fixed $t>0$. 
		
		\medskip
		{\noindent\em Case I: There is no $D_j$ intersecting $K$.}
		
		\medskip
		Clearly, $K$ is contained in one of $N_j$'s. Then the desired upper bound follows immediately.
		
		\medskip
		{\noindent\em Case II: $K$ intersects some $D_j$.}
		
		\medskip
		Then $K\subset \{x\in M;\dist_M(x,D_j)\leq 2\epsilon\}$. It follows that for any $x,y\in K$,
		\[ \dist_M(x,y)\leq \sup_{x_1,x_2\in D}\dist_M(x_1,x_2)+4\epsilon<\pi+\frac{8}{3}+4\epsilon,\]
		where the second inequality is from Lemma \ref{lem:stable fbms}.

		In particular, if $M\subset \mb R^3$, then $\mc U_S=\emptyset$. By Lemma \ref{lem:width for prime}, every connected component of level sets of $f_j$ has diameter upper bound 49. Therefore, such an $f$ satisfies our requirements and Theorem \ref{thm:f from M} is proved.
		
	\end{proof}

	\appendix

	\section{Gromov-Lawson's tricks}\label{sec:GL trick}
	In this section, we give the proof of Proposition \ref{prop:non-empty T_0} for $T_0\neq \emptyset$. The proof for $T_0=\emptyset$ is parallel and we omit the details.
	\begin{proof}[Proof of Proposition \ref{prop:non-empty T_0}]
		Since $(N, \partial_rN, T,g)$ is geometrically prime, then $T_0$ is connected. By Lemma \ref{lem:stable fbms}, for any $x_1,x_2\in T_0$,
		\begin{equation*}
		\dist_N(x_1,x_2)\leq \frac{59\pi}{3}+28.
		\end{equation*}
		For simplicity, let $x$ and $y$ denote $\gamma(0)$ and $\gamma(1)$, respectively. Then we have the following two cases. \underline{If $\mk d(x)\leq \frac{8\pi}{3}+4$}, 
		\begin{equation}\label{eq:small t}
		\dist_N(x,y)\leq \dist_N(x,T_0)+\frac{59\pi}{3}+28+\dist_N(y,T_0)\leq 25\pi+36,
		\end{equation}
		
		\underline{Now we assume that $\mk d(x)>\frac{8\pi}{3}+4+2\epsilon$ for some $\epsilon>0$}. Let $p_x$ and $p_y$ be the closest points in $T_0$ to $x$ and $y$, respectively. Let $\gamma_x$ (resp. $\gamma_y$) be the minimizing curve in $N$ from $p_x$ to $x$ (resp. $y$), i.e. 
		\[\mr{Length} (\gamma_x)=\mk d(x), \ \ \ \text{ and } \ \ \ \mr{Length}(\gamma_y)=\mk d(y).\]
		Let $\sigma$ denote a curve in $T_0$ connecting $p_x$ and $p_y$. Note that $\overline{\gamma_x\gamma{\gamma_y}^{-1}\sigma^{-1}}$ is a closed curve (denoted by $\wti \gamma$) on $N$. Since $N$ is geometrically prime, then $\wti \gamma$ bounds a surface $\Gamma' \subset N$ relative to $T_0$, i.e. $\partial\Gamma'\setminus T_0=\wti\gamma$. Let $\Gamma$ be an area minimizing surface among all those surfaces bounded by $\wti \gamma$ relative to $T_0$. Then $\mathrm{Int} (\Gamma)$ is a smoothly embedded stable minimal surface. By Theorem \ref{thm:compact boundary}, for any $x'\in \Gamma$, 
		\begin{equation}\label{eq:x to partial Gamma}
		\dist_N(x',\partial \Gamma)\leq \frac{4\pi}{3}+2.
		\end{equation}
		
		Next, we take $t^\prime$ such that
		\begin{enumerate}
			\item $\mk d(x)-\frac{4\pi}{3}-2-2\epsilon\leq t'\leq \mk d(x)-\frac{4\pi}{3}-2-\epsilon$;
			\item $\mk d^{-1}(t')$ is transverse to $\Gamma$, $\gamma_x$ and $\gamma_y$.
		\end{enumerate}
		Then there exists a curve $\alpha\subset\mk d^{-1}(t')\cap \Gamma$ joining $\gamma_x$ and $\gamma_y$. Recall that $\mk d(x)=\mk d(y)\leq \mk d(\gamma(s))$. This implies that 
		\[ \dist_N(\gamma,\alpha)\geq\inf_{x'\in \gamma, y'\in \alpha}\{\mk d(x')-\mk d(y')\}\geq \frac{4\pi}{3}+2+\epsilon.\]
		Together with (\ref{eq:x to partial Gamma}), we can find $z\in\alpha$ such that
		\[ \dist_\Gamma(z, \gamma_x)\leq \frac{4\pi}{3}+2, \ \ \text{ and } \ \ \dist_\Gamma(z, \gamma_y)\leq \frac{4\pi}{3}+2.\]
		Let $x_1$ and $y_1$ be the closest points to $z$ on $\gamma_x$ and $\gamma_y$ respectively, by triangle inequalities, we obtain that
		\[ \mk d(x_1)\geq t'-\frac{4\pi}{3}-2, \ \ \text{ and } \ \ \mk d(x_2)\geq t'-\frac{4\pi}{3}-2,\]
		and then
		\[ \dist_N(x,x_1)\leq \frac{8\pi}{3}+4+2\epsilon, \ \ \text{ and } \ \ \dist_N(y,y_1)\leq \frac{8\pi}{3}+4+2\epsilon.\]
		Therefore,
		\[\dist_N(x,y)\leq \dist_N(x,x_1)+\dist_N(x_1,z)+\dist_N(z,y_1)+\dist_N(y_1,y)\leq 8\pi+12+4\epsilon.\]
		Finally, combining with two cases above, we finish the proof of Proposition \ref{prop:non-empty T_0}.
		
	\end{proof}

	\bibliographystyle{amsalpha}
	\bibliography{minmax}

\end{document}